\documentclass[12pt]{amsart}
\usepackage{amssymb}

\newtheorem{theorem}{Theorem}[section]
\newtheorem{lemma}[theorem]{Lemma}
\newtheorem{proposition}[theorem]{Proposition}
\newtheorem{cor}[theorem]{Corollary}
\theoremstyle{definition}
\newtheorem{definition}[theorem]{Definition}
\newtheorem{example}[theorem]{Example}

\theoremstyle{remark}
\newtheorem{remark}[theorem]{Remark}
\theoremstyle{conjecture}

\theoremstyle{problem}
\newtheorem{problem}[theorem]{Problem}

\numberwithin{equation}{section}

\newcommand{\C}{\mathbb{C}}

\newcommand{\N}{\mathbb{N}}

\newcommand{\R}{\mathbb{R}}
\newcommand{\T}{\mathbb{T}}
\newcommand{\Z}{\mathbb{Z}}
\newcommand{\cG}{\mathcal{G}}
\newcommand{\cK}{\mathcal{K}}
\newcommand{\cL}{\mathcal{L}}
\newcommand{\cO}{\mathcal{O}}
\newcommand{\cT}{\mathcal{T}}

\newcommand{\cF}{\mathcal{F}}
\newcommand{\cM}{\mathcal{M}}

\newcommand{\fT}{\mathfrak{T}}
\newcommand{\fX}{\mathfrak{X}}
\newcommand{\fZ}{\mathfrak{Z}}

\newcommand\Tr{\operatorname{Tr}}

\newcommand\Ad{\operatorname{Ad}}

\newcommand\inpr[2]{\langle{#1,#2}\rangle}
\newcommand{\tM}{\widetilde{M}}
\newcommand{\tP}{\tilde{P}}
\newcommand{\tX}{\tilde{X}}
\newcommand{\htau}{\hat{\tau}}
\newcommand{\oo}{\overline{\omega}}

\newcommand{\ex}{\operatorname{ex}}

\title[The flow of weights and the Cuntz-Pimsner algebras]
{The flow of weights and the Cuntz-Pimsner algebras} 

\author{Masaki Izumi}
\address{Department of Mathematics\\ Graduate School of Science\\
Kyoto University\\ Sakyo-ku, Kyoto 606-8502\\ Japan}
\email{izumi@math.kyoto-u.ac.jp}

\subjclass[2010]{ 
Primary 46L30; Secondary 37A55}
\keywords{KMS states, Cuntz-Pimsner algebras, flow of weights, Poisson boundary, tail boundary}

\thanks{Supported in part by JSPS KAKENHI Grant Number JP15H03623}

\begin{document} 


\begin{abstract} We describe the flows of weights for von Neumann algebras 
arising from KMS states for the gauge actions of the Cuntz-Pimsner algebras in terms of 
Poisson (tail) boundaries. 
\end{abstract}

\maketitle

\section{Introduction} The Kubo-Martin-Schwinger (abbreviated as KMS) condition characterizes 
equilibrium states of quantum systems with infinitely many degrees of freedom. 
In the language of C$^*$-dynamical systems, it was first formulated by Haag-Hugenholtz-Winnink \cite{HHW67}, 
which laid the foundation for an operator algebraic approach to quantum statistical mechanics.  
Even apart from its relevance in physics, it also has a great influence over the purely mathematical theory of operator algebras.  
Indeed, soon after their work Takesaki \cite{T70} showed that the KMS condition uniquely characterizes 
the modular automorphism group of a given state of a von Neumann algebra, which is one of the most fundamental facts in Tomita-Takesaki theory. 
This paper is mainly concerned with purely operator algebraic features of KMS states. 

For a given $\R$-action on a C$^*$-algebra, it is well known that its KMS state extends to a faithful normal state on 
the weak closure of the image of the GNS representation, and the $\R$-action extends to 
the modular automorphism group (with an appropriate time scaling). 
This is still the most handy way to construct type III factors and compute their modular objects. 
On the other hand, Connes-Takesaki \cite{CT77} introduced the notion of the flow of weights for a von Neumann algebra, 
which turns out to be a complete invariant for the injective type III factors. 
Thus in a study of a specific class of KMS states, it would be desirable to compute the flows of weights for 
the von Neumann algebras arising from them in order to obtain more detailed information. 
One of the main purposes of this paper is to accomplish this task in the case of KMS states for the gauge actions of 
the Cuntz-Pimsner algebras, which have been extensively studied in these years 
(it is practically impossible to cite all the related references for this subject, and we content ourselves with an incomplete list,  
\cite{PWY00}, \cite{EL03}, \cite{LN04}, \cite{IKW07}, \cite{KW13}, \cite{LRRM14}, \cite{AHR14}, \cite{Kak15}). 
Although the structure of the simplex of $\beta$-KMS states of the gauge action have been well studied 
in the literature, the structure of the von Neumann algebras arising from them has been studied only in the restricted case of 
type III$_\lambda$ factors with $\lambda\neq 0$ (see \cite{OP78}, \cite{EFW84}, \cite{I93}, \cite{O02}, \cite{IKW07}). 

In general it is not so easy to concretely describe the flow of weights for a given type III factor unless the flow is transitive. 
For example, in the case of group measure space factors, the flows are given by so-called the associated flows, 
which are known to be difficult to compute except for special cases (see \cite{Kr76}, \cite{HOO75}, \cite{HO86}). 
In this paper, we follow the spirit of Connes-Woods \cite{CW89}, who showed that the flows for the Araki-Woods factors 
can be identified with the Poisson boundaries of time dependent random walks on the real line. 
Since the structure of the Cuntz-Pimsner algebras is somewhat more involved than that of the UHF algebras, 
our description of the corresponding random walks (or Markov operators) is not as direct as in the case of Araki-Woods factors.   

Since the gauge action is periodic, the flow for the von Neumann algebra $M$ arising from a $\beta$-KMS state of the gauge action 
is expressed as a suspension flow over a base transformation with a constant ceiling functions $|\beta|$. 
We construct a Markov operator whose Poisson boundary is identified with the center of $M$, 
and whose tail boundary together with the time translation is identified with the base transformation 
of the flow (Theorem \ref{main}). 
We show, by an easy example (Example \ref{III0}), that any probability-measure-preserving ergodic transformation can be realized 
as the base transformation of the flow, and hence show that type III$_0$ factors occur from the gauge actions 
of the Cuntz-Pimsner algebras. 
We also give examples constructed from correspondences whose left actions are degenerate, having rich structure 
of the KMS states of infinite type (Example \ref{non-unital}, Example \ref{GICAR}). 

The author would like to thank Hiroki Matui and Kengo Matsumoto for useful discussions on substitution 
dynamical systems. 
He also would like to thank the anonymous referee for careful reading, which in particular enables him 
to correct the statement of Theorem \ref{main} for negative $\beta$.   
\section{Infinite type KMS states}
Let $A$ be a $C^*$-algebra, and let $E$ be a C$^*$-correspondence of $A$, which is a right Hilbert $A$-module 
together with a homomorphism $\Phi$ from $A$ to the C$^*$-algebra $\cL(E)$ of bounded adjointable operators on $E$. 
For $\xi,\eta\in E$, we denote by $\theta_{\xi,\eta}\in \cL(E)$ the rank one operator defined by 
$\theta_{\xi,\eta}\zeta=\xi\cdot\inpr{\eta}{\zeta}_A$. 
We denote by $\cK(E)$ the closed ideal of $\cL(E)$ generated by the rank one operators, 
and set $I_E=\Phi^{-1}(\cK(E))$. 

For the definition of the Cuntz-Pimsner algebra $\cO_E$ for $E$, we adopt Pimsner's original one in \cite{P97}, 
which is enough for our purpose, though we do not assume that either $E$ is full, 
the left action $\Phi$ is non-degenerate, or it is injective (see Remark \ref{K} below).

\begin{definition}
The Cuntz-Toeplitz algebra $\cT_E$ for $E$ is the universal C$^*$-algebra generated by symbols $\{\Pi(a)\}_{a\in A}$ and 
$\{T_\xi\}_{\xi\in E}$ satisfying the following conditions:
\begin{itemize} 
\item the map $A\ni a\mapsto \Pi(a)\in \cT_E$ is a representation,  
\item $\Pi(a)T_\xi=T_{\Phi(a)\xi}$ for any $a\in A$ and $\xi\in E$, 
\item $T_\xi^*T_\eta=\Pi(\inpr{\xi}{\eta}_A)$ for any $\xi,\eta\in E$. 
\end{itemize} 
We denote by $\Psi$ the isomorphism from $\cK(E)$ into $\cT_E$ given by $\Psi(\theta_{\xi,\eta})=T_\xi T_\eta^*$ 
(see \cite{IKW07}). 
The Cuntz-Pimsner algebra $\cO_E$ for $E$ is the quotient of $\cT_E$ by the closed ideal 
generated by $\{\Pi(a)-\Psi\circ\Phi(a)\}_{a\in I_E}$. 
\end{definition}

It follows from the above definition that the map $E\ni\xi\mapsto T_\xi\in \cT_E$ is linear and the relation 
$T_\xi\Pi(a)=T_{\xi\cdot a}$ always holds for any $a\in A$ and $\xi\in E$. 
The full Fock space construction shows that $\Pi$ is always faithful, and in consequence 
we have $\|T_\xi\|=\|\xi\|$ for any $\xi\in E$. 
The algebra $\cO_E$ may be small without additional assumptions on the correspondence $E$. 

Let $\gamma$ be the gauge action of $\cT_E$, that is an $\R$-action on $\cT_E$ given by $\gamma_t(T_\xi)=e^{t\sqrt{-1}}T_\xi$ and 
$\gamma_t(\Pi(a))=\Pi(a)$ for any $\xi\in E$, $a\in A$, and $t\in \R$. 
We denote by $K_\beta(\gamma)$ the set of $\beta$-KMS states of $\gamma$, and by $\ex K_\beta(\gamma)$ the set 
of extreme points in $K_\beta(\gamma)$. 
In what follows, we fix $\omega\in K_\beta(\gamma)$ with $\beta\neq 0,\infty$. 
We denote by $(\pi_\omega,H_\omega,\Omega_\omega)$ the GNS triple for $\omega$, 
and set $M=\pi_\omega(\cT_E)''$, $B=\pi_\omega(\Pi(A))''$. 
We denote by $\oo$ the state of $M$ given by $\oo(x)=\inpr{x\Omega_\omega}{\Omega_\omega}$. 
Then $\oo$ is a faithful normal state of $M$ whose modular automorphism group satisfies 
$\sigma^{\oo}_t(\pi_\omega(x))=\pi_\omega(\gamma_{-\beta t}(x))$. 
Our main goal in this paper is to determine the structure of $M$.

We set $\tau(a)=\omega(\Pi(a))$ for $a\in A$, which is a finite trace of $A$, 
though it may not be a state as we do not assume that $\Phi$ is non-degenerate. 
Note that the representation $\pi_\omega\circ \Pi$ of $A$ is quasi-equivalent to $\pi_\tau$ (see \cite[Lemma 4.1]{I93}). 
Thus the weak closure $B_0=\overline{\pi_\omega(\Pi(A))}^w$ is isomorphic to $\pi_\tau(A)''$, which is a finite von Neumann algebra. 
When $\tau$ is a state (e.g. when $\Phi$ is non-degenerate), we have $B_0=B$.   
When $\tau$ is not a state, the representation $\pi_\omega\circ \Pi$ is degenerate, and 
$B$ is isomorphic to $B_0\oplus \C$. 
Indeed, the unit of $\overline{\pi_\omega(\Pi(A))}^w$ is the limit of 
the net $\{\pi_\omega\circ\Pi(a_i)\}_{i\in I}$ in the strong operator topology, 
where $\{a_i\}_{i\in I}$ is an approximate unit of $A$. 
We set 
\begin{equation}\label{p-infty}
p_\infty=1-s\mathchar`-\lim_{i\in I}\pi_\omega\circ\Pi(a_i),
\end{equation}
which is a projection in $B$. 
Then  
$$\oo(p_\infty)=1-\lim_{i\in I}\omega\circ \Pi(a_i)=1-\lim_{i\in I}\tau(a_i)=1-\|\tau\|.$$
Since $\oo$ is faithful, we see that $p_\infty=0$ if and only if $\tau$ is a state. 
Thus when $\tau$ is not a state, we have $B=B_0\oplus \C p_\infty$.

We may and do choose a subset $\Lambda\subset E$ with the following property:  
the net $\{\sum_{\xi \in F}\theta_{\xi,\xi}\}_{F\Subset \Lambda}$ is an approximate unit of $\cK(E)$, 
where $F\Subset \Lambda$ means that $F$ is a finite subset of $\Lambda$, and the set of finite subsets 
of $\Lambda$ is directed by the inclusion relation. 
In the sequel, whenever we use the symbol $\sum_{\xi\in \Lambda}$, it means the limit of partial sums with respect to 
this directed set.   
Since $\{\sum_{\xi\in F}T_\xi T_\xi^*\}_{F\Subset \Lambda}$ is an approximate unit of the C$^*$-subalgebra 
of $\cT_E$ generated by $\{T_\xi T_\eta^*\}_{\xi,\eta}$, the summation $\sum_{\xi\in \Lambda}\omega(T_\xi T_\xi^*)$ 
converges and is bounded by 1, which does not depend on the choice of $\Lambda$. 

The following definition is consistent with that in \cite{LN04} and \cite{IKW07} 
when $\Phi$ is non-degenerate (see \cite[Lemma 3.6]{IKW07}). 

\begin{definition} We say that $\omega\in K_\beta(\gamma)$ is of infinite type if 
\begin{equation}\label{infinite}
\sum_{\xi\in \Lambda}\omega(T_\xi T_\xi^*)=1.
\end{equation}
We denote by $K_\beta^\infty(\gamma)$ the set of $\beta$-KMS state of $\gamma$ of infinite type. 
\end{definition}

\begin{remark} The word ``type" in the above definition should not be confused with the type of the 
von Neumann algebra $M$. 
\end{remark}

The set $K_\beta^\infty(\gamma)$ is a face of $K_\beta(\gamma)$ in the sense that if  
$\omega\in K_\beta^\infty(\gamma)$ is expressed as $\omega=t\omega_1+(1-t)\omega_2$ with 
$\omega_1,\omega_2\in K_\beta(\gamma)$ and $0<t<1$, then $\omega_1,\omega_2\in K_\beta^\infty(\gamma)$. 
When $\Lambda$ is a finite set, namely $E$ is a finitely generated projective module of $A$, 
then $K_\beta^\infty(\gamma)$ is a closed subset of $K_\beta(\gamma)$ in the weak$*$ topology.

The following lemma is essentially \cite[Theorem 3.7]{IKW07}, which says that 
our main interest is KMS states of infinite type. 

\begin{lemma} Let $\omega\in \ex K_\beta(\gamma)$ be not of infinite type. 
Then $M$ is Morita equivalent to a corner of $\pi_\tau(A)''$, and in particular, it is semi-finite.  
\end{lemma}

\begin{proof} Since $\{\sum_{\xi\in F}T_\xi T_\xi^*\}_{F\Subset \Lambda}$ is an approximate unit of the 
C$^*$-subalgebra of $\cT_E$ generated by $\{T_\xi T_\eta^*\}_{\xi,\eta\in E}$, the net 
$\{\sum_{\xi\in F}\pi_\omega(T_\xi T_\xi^*)\}_{F\Subset \Lambda}$ converges to a projection in the strong operator 
topology, and we set 
$$q_\infty=1-s\mathchar`-\sum_{\xi\in \Lambda}\pi_\omega(T_\xi T_\xi^*).$$
Since $\omega$ is not of infinite type, it is a non-zero projection in $M$ satisfying $q_\infty\pi_\omega(T_\xi)=0$ 
for any $\xi\in E$. 
We claim that $q_\infty\in B'$. 
Indeed, for $\xi\in E$ and $a\in A$, we have 
$$q_\infty \pi_\omega(\Pi(a))\pi_\omega(T_\xi)=q_\infty\pi_\omega(T_{\Phi(a)\xi})=0,$$
and 
$$q_\infty\pi_\omega(\Pi(a))(1-q_\infty)
=\lim_{F\Subset \Lambda}q_\infty\pi_\omega(\Pi(a))\sum_{\xi\in F}\pi_\omega(T_\xi T_\xi^*)=0,$$
which shows the claim. 

Note that since $\omega$ is an extreme KMS state, the von Neumann algebra $M$ is a factor, and it is Morita 
equivalent to $q_\infty M q_\infty$. 
On the other hand, since $q_\infty \pi_\omega(T_\xi)=0$, we get 
$$q_\infty M q_\infty=q_\infty \overline{\pi_\omega(\Pi(A))}^w q_\infty=q_\infty \overline{\pi_\omega(\Pi(A))}^w,$$
which shows the statement. 
\end{proof}

From now on we assume that $\omega$ is of infinite type. 

\begin{lemma} The state $\omega$ and the representation $\pi_\omega$ factor through $\cO_E$. 
\end{lemma}

\begin{proof} Let $a\in I_E=\Phi^{-1}(\cK(E))$. 
Since $\oo$ is faithful, it suffices to show 
$$\omega((\Pi(a)-\Psi\circ \Phi(a))^*(\Pi(a)-\Psi\circ \Phi(a)))=0.$$
Note that we have $\Psi\circ \Phi(a)T_\xi=T_{\Phi(a)\xi}$. 
Since $\{\sum_{\xi\in F}T_\xi T_\xi^*\}_{F\Subset \Lambda}$ is an approximate unit of the C$^*$-subalgebra 
of $\cT_E$ generated by $\{T_\xi T_\eta^*\}_{\xi,\eta}$, we have 
\begin{align*}
\lefteqn{(\Pi(a)-\Psi\circ \Phi(a))^*(\Pi(a)-\Psi\circ \Phi(a)))} \\
 &=\Pi(a^*a)-\Pi(a^*)\Psi\circ\Phi(a)-\Psi\circ \Phi(a^*)\Pi(a)+\Psi\circ \Phi(a^*a) \\
 &=\Pi(a^*a)+\sum_{\xi\in \Lambda}(-\Pi(a^*)\Psi\circ\Phi(a)T_\xi T_\xi^*
 -T_\xi T_\xi^*\Psi\circ \Phi(a^*)\Pi(a)
 +\Psi\circ \Phi(a^*a)T_\xi T_\xi^*)\\
  &=\Pi(a^*a)+\sum_{\xi\in \Lambda}(-\Pi(a^*)T_{\Phi(a)\xi} T_\xi^*
 -T_\xi T_{\Phi(a)\xi}^*\Pi(a)
 +T_{\Phi(a^*a)\xi} T_\xi^*)\\
 &=\Pi(a^*a)+\sum_{\xi\in \Lambda}(-\Pi(a^*a)T_{\xi} T_\xi^*
 -T_\xi T_{\xi}^*\Pi(a^*a)
 +\Pi(a^*a)T_\xi T_\xi^*)\\
 &=\Pi(a^*a)-\sum_{\xi\in \Lambda}T_\xi T_{\xi}^*\Pi(a^*a),
\end{align*}
where convergence is in the norm topology. 
On the other hand, since $\omega$ is of infinite type, the net $\{\sum_{\xi\in F}\pi_\omega(T_\xi T_\xi^*)\}_{F\Subset \Lambda}$ 
converges to 1 in the strong operator topology. 
Thus we get 
\begin{align*}
\lefteqn{\omega((\Pi(a)-\Psi\circ \Phi(a))^*(\Pi(a)-\Psi\circ \Phi(a)))} \\
 &=\omega(\Pi(a^*a))-
 \inpr{\pi_\omega(\sum_{\xi\in \Lambda}T_\xi T_{\xi}^*)\pi_\omega(\Pi(a^*a))\Omega_\omega}{\Omega_\omega}\\
 &=0,
\end{align*}
which shows the statement. 
\end{proof}

\begin{remark}\label{K} In order to include broader classes of examples, Katsura \cite{Kat03} modified the definition of 
the Cuntz-Pimsner algebras, and our $\cO_E$ is a quotient of the modified one. 
Since the GNS representation of an infinite KMS state of $\cT_E$ always factors through $\cO_E$, 
the original definition is good enough for our purpose. 
\end{remark}

For $\omega$ of infinite type, we have
$$\tau(a)=\sum_{\xi\in \Lambda}\omega(\Pi(a)T_\xi T_\xi^*)=e^{-\beta}\sum_{\xi\in \Lambda}\omega(T_\xi^*T_{\Phi(a)\xi})
=e^{-\beta}\sum_{\xi\in \Lambda}\tau(\inpr{\xi}{\Phi(a)\xi}_A).$$
On the other hand, the left-hand side of the condition (\ref{infinite}) is
$$e^{-\beta}\sum_{\xi\in \Lambda}\omega(T_\xi^* T_\xi)
=e^{-\beta}\sum_{\xi\in \Lambda}\omega(\Pi(\inpr{\xi}{\xi}_A))
=e^{-\beta}\sum_{\xi\in \Lambda}\tau(\inpr{\xi}{\xi}_A).$$
Thus an easy modification of the proof in \cite[Theorem 2.1]{LN04} with these conditions show the following theorem. 

\begin{theorem}\label{KMS0} The restriction map $\omega\mapsto \tau=\omega\circ \Pi$ gives an affine isomorphism of 
$K_\beta^\infty(\gamma)$ onto the set of finite traces $\tau$ of $A$ satisfying the following two conditions: 
\begin{equation}\label{KMS1}
e^{-\beta}\sum_{\xi\in \Lambda}\tau(\inpr{\xi}{\xi}_A)=1,
\end{equation}
\begin{equation}\label{KMS2}
e^{-\beta}\sum_{\xi\in \Lambda}\tau(\inpr{\xi}{\Phi(a)\xi}_A)=\tau(a).
\end{equation}
When $\Phi$ is non-degenerate, the above $\tau$ is necessarily a trace state. 
\end{theorem}

\section{The flow of weights}
We keep using the same notation as in the previous section. 
Recall that $\omega$ is a $\beta$-KMS state of infinite type for the gauge action $\gamma$ of the Cuntz-Toeplitz algebra $\cT_E$, or rather 
the Cuntz-Pimsner algebra $\cO_E$. 
For simplicity, we use the notation $S_\xi=\pi_\omega(T_\xi)$ for $\xi\in E$, $\pi(a)=\pi_\omega(\Pi(a))$ for $a\in A$, and $\varphi=\oo$, which is 
the weakly continuous extension of $\omega$ to the weak closure $M$ of $\cT_E$ in the GNS representation for $\omega$. 
The modular automorphism group $\{\sigma^\varphi_t\}_{t\in \R}$ is the weakly continuous extension of $\{\gamma_{-\beta t}\}_{t\in \R}$ to $M$. 
Since $\omega$ is of infinite type, we have the convergence 
\begin{equation}\label{E1}
s\mathchar`-\sum_{\xi\in \Lambda}S_\xi S_\xi^*=1,
\end{equation}
in the strong operator topology. 
For a von Neumann algebra $N$, we denote by $\fZ(N)$ the center of $N$. 

Let $\widehat{\sigma^\varphi}$ be the dual action of the modular automorphism group $\{\sigma^\varphi_t\}_{t\in \R}$, 
acting on the crossed product $M\rtimes_{\sigma^\varphi}\R$. 
Then the (smooth) flow of weights for $M$ is the point realization of the restriction of $\widehat{\sigma^\varphi}_t$ 
to the center $\fZ(M\rtimes_{\sigma^\varphi}\R)$, which depends only on the isomorphism class of $M$ (see \cite{CT77}). 
The flow is ergodic if and only if $M$ is a factor, and it is dissipative if and only if $M$ is semi-finite. 

In our particular case, the modular automorphism group $\{\sigma^\varphi_t\}_{t\in \R}$ is periodic with period $T:=2\pi/|\beta|$, 
and we have a little simpler description of the flow. 
We set 
$$\tM=M\rtimes_{\sigma^\varphi}\R/T\R,$$ 
which is a von Neumann algebra generated by $M$ and the implementing 
unitary representation $\{\lambda(t)\}_{t\in \R/T\R}$ of the group $\R/T\R$. 
Let $\theta$ be the dual action, which is given by $\theta(x)=x$ for 
$x\in M$ and $\theta(\lambda(t))=e^{-|\beta| t\sqrt{-1}}\lambda(t)$.  
Since $\fZ(\tM)$ is an abelian von Neumann algebra, we can identify it with $L^\infty(\fX_M,m_M)$ 
for a measure space $(\fX,m_M)$, and $\theta|_{\fZ(\tM)}$ is realized by a non-singular transformation 
$\fT_M$ on $(\fX_M,m_M)$ as $\theta(f)(x)=f(\fT_M^{-1}x)$. 

\begin{definition} We call $\fT_M$ the reduced flow of weights for $M$. 
\end{definition}

We often abuse the terminology, and call the pair $(\fZ(\tM),\theta|_{\fZ(\tM)})$ the reduced flow of weights too. 
It is easy to show the following (see \cite[Section 10]{T73}). 

\begin{lemma} The flow of weights for $M$ is a suspension flow over a base transformation $\fT_M$ 
with a constant ceiling function $|\beta|$. 
More precisely, the flow space is $\fX\times [0,|\beta|)$ and the flow is given by 
$$t\cdot (x,s)=(\fT_M^n x,s+t-|\beta|n),$$
where $n$ is an integer satisfying $0\leq s+ t-n|\beta|<|\beta|$.  

In consequence, 
\begin{itemize}
\item[(1)] $\fT_M$ is ergodic if and only if $M$ is a factor.
\item[(2)] $\fT_M$ is dissipative if and only if $M$ is semi-finite. 
\item[(3)] $\fT_M$ has a finite invariant measure if and only if the flow of weights for $M$ has a finite invariant measure.  
\end{itemize}
\end{lemma}

Now out task is to determine the reduced flow of weights for $M$. 
For this, we introduce a Markov operator on the abelian von Neumann algebra $\fZ(B)$. 
A Markov operator acting on a von Neumann algebra is a unital normal completely positive map. 
We denote by $E_{\fZ(B)}$ the $\varphi$-preserving conditional expectation from 
$M$ onto $\fZ(B)$, which uniquely exists because the modular automorphism group of $\varphi$ 
leaves $\fZ(B)$ invariant. 
Thanks to Eq.(\ref{E1}), we can introduce a Markov operator on $\fZ(B)$ by 
$$P(f)=E_{\fZ(B)}(s\mathchar`-\sum_{\xi\in \Lambda}S_\xi f S_\xi^*).$$ 
We take a point realization of $\fZ(B)$, and identify $\fZ(B)$ with $L^\infty(X,\mu)$. 
Since $\varphi$ is faithful, we may assume that the state $\varphi$ restricted to $\fZ(B)$ 
is given by the probability measure $\mu$. 

We briefly recall the basics of the Poisson (and tail) boundaries, which will be used for our 
description of the reduced flow of weights.  
The reader is referred to \cite{Kai92} for a purely measure theoretical construction, and to \cite{I02}, \cite{I12} 
for operator algebraic treatments. 

\begin{definition}
A function $f\in L^\infty(X,\mu)$ is said to be $P$-harmonic if $P(f)=f$. 
We denote by $H^\infty((X,\mu),P)$ the set of $P$-harmonic functions, 
which inherits operator system structure from $L^\infty(X,\mu)$. 

A bounded sequence $\{f_n\}_{n\in \Z}$ in $L^\infty(X,\mu)$ 
is said to be a $P$-harmonic sequence if $P(f_n)=f_{n-1}$. 
We denote by $S^\infty((X,\mu),P)$ the set of bounded harmonic sequences, 
which inherits operator system structure from $L^\infty(X,\mu)\otimes \ell^\infty(\Z)$. 
\end{definition} 

It is known that the operator system $H^\infty((X,\mu),P)$ has a commutative von Neumann 
algebra structure with a new product given by 
$$s\mathchar`-\lim_{n\to \infty}P^n(fg),$$
and the point realization of this abelian von Neumann algebra is said to be the Poisson boundary for $P$, 
and denoted by $(\overline{\Omega},\overline{\nu})$. 
Since the map 
$$H^\infty((X,\mu),P)\ni f\mapsto \mu(f)\in \C$$ 
gives rise to a faithful normal state as a von Neumann algebra, 
we may assume that the measure $\overline{\nu}$ comes from this state. 

Let $\tP$ be a Markov operator on $L^\infty(X\times \Z,\mu\times m)$ 
defined by 
$$\tP(f)(x,n)=P(f(\cdot,n+1))(x),\quad x\in X,$$
where $m$ is the counting measure on $\Z$. 
Then the space $S^\infty((X,\mu),P)$ is identified with $H^\infty((X\times \Z,\mu\times m),\tP)$. 
The Poisson boundary for $\tP$ is called the tail boundary for $P$ and denoted by $(\Omega,\nu)$. 
Since $P$ is a faithful normal positive map, the map $(f_n)\mapsto \mu(f_0)$ 
gives a faithful normal state on $L^\infty(\Omega,\nu)$. 
Thus we may assume that the map $(f_n)\mapsto \mu(f_0)$ corresponds 
to the measure $\nu$ as before. 
Let $\Sigma$ be the translation operator on $L^\infty(X\times \Z,\mu\times m)$ defined by 
$$\Sigma(f)(x,n)=f(x,n-1).$$
Then $\Sigma$ induces an automorphism of $L^\infty(\Omega,\nu)$, which is also denoted by $\Sigma$ for simplicity. 
Note that $\Sigma$ is $\nu$-preserving if $\mu\cdot P=P$ holds. 
The Poisson boundary $(\overline{\Omega},\overline{\nu})$ 
can be identified with the factor space of the tail boundary 
$(\Omega,\nu)$ by the translation $\Sigma$, namely 
$$L^\infty(\overline{\Omega},\overline{\nu})=L^\infty(\Omega,\nu)^\Sigma,$$
where the right-hand side means the set of functions in $L^\infty(\Omega,\nu)$ fixed by $\Sigma$. 

\begin{remark} For the definition of the tail boundary, it does not matter whether we consider 
one-sided sequences or two sided-sequences because they give isomorphic operator systems. 
\end{remark}

\begin{theorem}\label{main}
The reduced flow of weights $(\fZ(\tM),\theta|_{\fZ(\tM)})$ is identified with the tail boundary for $P$ with 
the time translation $(L^\infty(\Omega,\nu),\Sigma)$ for positive $\beta$, and with $\Sigma^{-1}$ instead of 
$\Sigma$ for negative $\beta$.  
The center $\fZ(M)$ of $M$ is identified with the Poisson boundary  
$L^\infty(\overline{\Omega},\overline{\nu})=L^\infty(\Omega,\nu)^\Sigma$ for $P$. 
\end{theorem}

We prove Theorem \ref{main} in several steps. 
We first give a characterization of the Markov operator $P$, which also provides a way to 
compute $P$ in concrete examples. 
For any finite subset $F\Subset \Lambda$, we introduce a normal completely positive map $Q_F$ acting on $M$ by 
$$Q_F(x)=e^{-\beta}\sum_{\xi\in F}S_\xi^*xS_\xi,$$
which globally preserves $B$. 
The KMS condition implies $\varphi\circ Q_F\leq \varphi$. 
It would be very convenient for us if we could introduce $Q$ by the limit of the net $\{Q_F\}_{F\Subset \Lambda}$, 
which would play the role of the adjoint of $P$, though a priori there is no reason 
for the desired limit to exist. 
Note that for $a\in A$, we have 
$$Q_F(\pi(a))=e^{-\beta}\sum_{\xi\in F}\pi(\inpr{\xi}{\Phi(a)\xi}_A).$$
When the net $\{\sum_{\xi\in F}\pi(\inpr{\xi}{\xi}_A)\}_{F\Subset \Lambda}$ is bounded, 
we can directly define 
$$Q(x)= e^{-\beta}\;s\mathchar`-\sum_{\xi\in \Lambda}S_\xi^*xS_\xi.$$
Thanks to (\ref{KMS2}), we have $\varphi \circ Q=\varphi$. 

\begin{lemma}\label{trans} For any $f\in \fZ(B)$ and $g\in B$, we have 
$$\varphi(P(f)g)=\lim_{F\Subset \Lambda} \varphi(fQ_F(g)).$$
In particular, when $Q$ is well-defined, we have 
$$\varphi(P(f)g)=\varphi(fQ(g)).$$
\end{lemma}

\begin{proof} Recall that $p_\infty$ is a central projection in $B$ defined by (\ref{p-infty}), 
which is minimal in $B$ if it is not zero. 
Since $\{T_\xi\cdot a_i\}_{i\in I}$ converges to $T_\xi$ for an approximate unit $\{a_i\}_{i\in I}$ of $A$, 
we have $S_\xi p_\infty=0$, and 
$$P(p_\infty)=E_{\fZ(B)}(s\mathchar`-\sum_{\xi\in \Lambda}(S_\xi p_\infty S_\xi^*))=0,$$
$$\varphi(p_\infty Q_F(g))=\sum_{\xi\in F}\varphi(p_\infty S_\xi^* g S_\xi)=0.$$
Thus the statement for $f=p_\infty$ holds. 

Now assume $f\in \fZ(B_0)$ with $0\leq f\leq 1$. 
In a similar way as in \cite[Theorem 1.1]{LN04}, we can show that the map 
$$B\ni g\mapsto \lim_{F\Subset \Lambda}\varphi(f Q_F(g))$$ is a trace of $B$ dominated by $\varphi|_B$, 
and there exists a unique element $h\in \fZ(B)$ satisfying  
$$\lim_{F\Subset \Lambda}\varphi (fQ_F(g))=\varphi(hg)$$ 
for all $g\in B$. 
On the other hand, the KMS condition implies  
$$\varphi(fQ_F(g))=e^{-\beta}\sum_{\xi\in F}\varphi(fS_\xi^*gS_\xi)
=\sum_{\xi\in F}\varphi(S_\xi fS_\xi^*g)
=\sum_{\xi\in F}\varphi(E_B(S_\xi fS_\xi^*)g),$$
where $E_B$ is the normal $\varphi$-preserving conditional expectation from $M$ onto $B$. 
This implies 
$$\varphi(hg)=\lim_{F\Subset \Lambda}\varphi(fQ_F(g))
=\lim_{F\Subset \Lambda}\sum_{\xi\in F}\varphi(E_B(S_\xi fS_\xi^*)g)
=\varphi(E_B(s\mathchar`-\sum_{\xi\in \Lambda}S_\xi f S_\xi^*)g),$$
and we get $P(f)=h$. 
Thus the statement holds. 
\end{proof}

Let 
$$\lambda(t)=\sum_{n\in \Z}e^{n\beta t\sqrt{-1}}e_n$$
be the spectral decomposition of $\{\lambda(t)\}_{t\in \R/T\R}$. 
The relation $\lambda(t)S_\xi=e^{-\beta t\sqrt{-1}}S_\xi\lambda(t)$ for $\xi\in E$ implies 
\begin{equation}\label{E2}
e_nS_\xi=S_\xi e_{n+1},\quad \xi\in E.
\end{equation}
We set 
$$h=\sum_{n\in \Z}e^{n\beta}e_n,$$
which is a positive non-singular operator affiliated with $\tM$.  
Let $\hat{\varphi}$ be the dual weight of $\varphi$ defined by
$$\hat{\varphi}(x)=\varphi(\sum_{n\in \Z}\theta^n(x)),\quad x\in \tM_+,$$ 
and let $\htau=\hat{\varphi}(h^{-1}\cdot)$. 
Then $\htau$ is a faithful normal semi-finite trace satisfying 
$\htau\circ\theta=e^{-|\beta|}\htau$ and
\begin{equation}\label{E3}
\htau(xe_n)=\varphi(x)e^{-n\beta},\quad x\in M. 
\end{equation} 

To determine the structure of $\fZ(\tM)$, we investigate 
the set of normal semi-finite traces on $\tM$ 
dominated by a multiple of $\htau$ because the latter set 
completely determines the order structure of $\fZ(\tM)_+$. 

Let $\psi$ be a normal semi-finite trace dominated by $\htau$. 
Since $\psi(e_n)\leq \htau(e_n)<\infty$, for every $x\in \tM_+$, 
we have 
$$\psi(x)=\sum_{n\in \Z}\psi(x^{1/2}e_nx^{1/2})
=\sum_{n\in \Z}\psi(e_nxe_n).$$
It is straightforward to show $e_n\tM e_n=M_\varphi e_n$, 
and so $\psi$ is determined by the family of finite traces 
$M_\varphi\ni x\mapsto \psi(xe_n)$ for $n\in \Z$. 

For $\xi=(\xi_1,\xi_2,\ldots,\xi_n)\in E^n$, we set $S_\xi=S_{\xi_1}S_{\xi_2}\cdots S_{\xi_n}$. 
Let $\cF_k^0$ be the liner span of $S_\xi S_\eta^*$ for all 
$\xi,\eta\in E^k$, with convention $\cF_0^0=\pi(A)$, 
and let $\cF^0=\cup_{k=0}^\infty \cF_k^0$. 
Then $\cF^0$ is a dense $*$-subalgebra of $M_\varphi$.  
Thanks to Eq.(\ref{E1}), we have inclusion 
${\cF_k^0}''\subset {\cF_{k+1}^0}''$. 
 
For $\xi,\eta\in E^k$,
$$\psi(S_\xi S_\eta^*e_n)=\psi(S_\eta^*e_n S_\xi)
=\psi(S_\eta^*S_\xi e_{n+k}).$$
Thus $\psi$ is completely determined by the family of finite traces 
$B \ni x\mapsto \psi(xe_n)$, which is dominated by $e^{-n\beta}\varphi|_B$.  
There exists a unique element $f_n\in \fZ(B)$ with 
$0\leq f_n\leq 1$ satisfying 
$$\psi(xe_n)=\varphi(f_nx)e^{-n\beta},\quad x\in B.$$

\begin{lemma} The sequence $(f_n)$ belongs to $S^\infty((X,\mu),P)$. 
\end{lemma}

\begin{proof}
Thanks to (\ref{E1}), for any $x\in B$ we have
\begin{eqnarray*}\psi(xe_n)&=&
\sum_{\xi\in \Lambda}^\infty\psi(S_\xi S_\xi^*xe_n)
=\sum_{\xi\in \Lambda}^\infty\psi(S_\xi^*xe_nS_\xi)
=\sum_{\xi\in \Lambda}^\infty\psi(S_\xi^*x S_\xi e_{n+1})\\
&=&e^\beta\lim_{F\Subset \Lambda}\psi(Q_F(x)e_{n+1})
=\lim_{F\Subset \Lambda}\varphi(f_{n+1}Q_F(x))e^{-n\beta}\\
&=&\varphi(P(f_{n+1})x)e^{-n\beta},
\end{eqnarray*}
which implies $P(f_{n+1})=f_n$. 
\end{proof}

Note that the above correspondence $\psi\mapsto (f_n)$ gives rise to 
an injective order-preserving affine map 
$$\rho:\fZ(\tM)_+\rightarrow L^\infty(\Omega,\nu)_+.$$
Indeed, starting from $z\in \fZ(\tM)_+$, we can construct a normal semi-finite trace $\psi$ of $\tM$ by 
$\psi(x)=\htau(zx)$, which is dominated by a multiple of $\htau$. 
Applying the above argument to $\psi$, we get $(f_n)\in S^\infty((X,\mu),P)$. 
We define $\rho(z)$ to be the element in $L^\infty(\Omega,\nu)$ corresponding to $(f_n)\in S^\infty((X,\mu),P)$. 

By construction, we have $\rho(1)=1$. 
Note that $\theta(e_n)=e_{n+\epsilon}$ holds with $\epsilon=\beta/|\beta|$. 
For $z\in \fZ(\tM)_+$ with $\htau(zxe_n)=\varphi(xf_n)e^{-n\beta}$, we have 
\begin{align*}
\lefteqn{\htau(\theta(z)xe_n)=e^{-|\beta|}\htau(x\theta^{-1}(e_n))=e^{-|\beta|}\htau(xe_{n-\epsilon})} \\
 &=e^{-|\beta|}\varphi(xf_{n-\epsilon})e^{-(n-\epsilon)\beta}=\varphi(xf_{n-\epsilon})e^{-n\beta}.
\end{align*}
Thus $\rho$ satisfies $\rho\circ \theta=\Sigma\circ\rho$ for positive $\beta$ and 
$\rho\circ \theta=\Sigma^{-1}\circ\rho$ for negative $\beta$. 

Now Theorem \ref{main} follows from the next lemma. 

\begin{lemma} Let $f=(f_n)$ be a bounded harmonic sequence 
for $P$ with $0\leq f_n\leq 1$ for all $n\in \Z$. 
Then there exists a unique normal semi-finite trace 
$\psi$ on $\tM$ dominated by $\htau$ satisfying 
$$\psi(xe_n)=\varphi(f_nx)e^{-n\beta},\quad x\in B.$$
The correspondence $(f_n)\mapsto \psi$ is order-preserving. 
\end{lemma} 
 
\begin{proof} 
First we construct a family of bounded normal traces $\psi_n$ on $M_\varphi$. 
Note that for any $x\in (M_\varphi)_+$ and $F_i\Subset \Lambda$, $i=1,2,\ldots,k$, we have an inequality
$$\varphi(f_{n+k} Q_{F_1}\circ Q_{F_2}\circ \cdots Q_{F_k}(x))\leq \varphi(Q_{F_1}\circ Q_{F_2}\circ \cdots Q_{F_k}(x))\leq\varphi(x).$$
Thus the limit 
$$\psi_{n,k}(x)=\lim_{F_k\Subset \Lambda}\cdots \lim_{F_2\Subset \Lambda}\lim_{F_1\Subset \Lambda}
\varphi(f_{n+k} Q_{F_1}\circ Q_{F_2}\circ \cdots Q_{F_k}(x))$$
exists and $\psi_{n,k}$ extends to a positive functional on $M_\varphi$ dominated by $\varphi$. 
We choose a free ultrafilter $\kappa\in \beta\N\setminus \N$ and set 
$$\psi_{n}(x)=\lim_{k\to\kappa}\psi_{n,k}(x),\quad x\in M_\varphi.$$
Then $\psi_n$ is a positive linear functional in $M_\varphi^*$ 
dominated by $\varphi|_{M_\varphi}$, and so $\psi_{n}$ is normal. 
Thanks to Lemma \ref{trans}, we have $\psi_{n,k}(y)=\varphi(P^k(f_{n+k})y)=\varphi(f_n y)$ 
for all $y\in B$ and $\psi_n(y)=\varphi(f_n y)$. 
In particular, the restriction of $\psi_n$ to $B$ is a trace. 

Direct computation implies that for all $\xi,\eta\in E^l$ and $l\leq k$,  
$$\psi_{n,k}(S_\xi S_\eta^*)=e^{-l\beta}\varphi(f_{n+l}S_\eta^*S_\xi),$$
which implies 
\begin{equation}\label{E4}
\psi_n(S_\xi S_\eta^*)
=e^{-l\beta}\varphi(f_{n+l}S_\eta^*S_\xi).
\end{equation}
Since  $\psi_n$ is normal and $\cF^0$ is dense in $M_\varphi$, 
this shows that $\psi_n$ does not depend on the choice of the ultrafilter, and the correspondence 
$(f_n)\mapsto (\psi_n)$ is order-preserving. 

Let $\xi_1,\xi_2,\eta_1,\eta_2\in E^l$. 
Eq.(\ref{E4}) implies 
$$\psi_n(S_{\xi_1}S_{\eta_1}^*S_{\xi_2}S_{\eta_2}^*)
=e^{-\beta l}\varphi(f_{n+l}(S_{\eta_2}^*S_{\xi_1})(S_{\eta_1}^*S_{\xi_2})),$$
and $\psi_n$ is a trace. 
Therefore there exists a unique element $g_n\in \fZ(M_\varphi)$ satisfying 
$0\leq g_n\leq 1$ and $\psi_n(x)=\varphi(g_nx)$ for all $x\in M_\varphi$. 
Similar computation as above shows 
$$\psi_n(S_\xi x S_\eta^*)=e^{-\beta}\psi_{n+1}(S_\eta^*S_\xi x)$$
for all $\xi,\eta\in E$ and $x\in M_\varphi$,  
and we get 
$$S_\eta^*g_nS_\xi=S_\eta^*S_\xi g_{n+1},\quad \forall \xi,\eta\in E.$$
Thus 
$$g_nS_\xi=\sum_{\eta\in \Lambda}^\infty S_\eta S_\eta^* g_nS_\xi
=\sum_{\eta\in \Lambda}^\infty S_\eta S_\eta^* S_\xi g_{n+1}=S_\xi g_{n+1}.$$
This together with Eq.(\ref{E2}) implies $g_ne_nS_\xi=S_\xi g_{n+1}e_{n+1}$, and 
$$g=\sum_{n\in \Z}g_ne_n$$
belongs to $\fZ(\tM)$. 

Let $\psi=\htau(g\cdot)$. 
Then $\psi$ is a normal semi-finite trace on $\tM$ dominated by $\htau$. 
For $x\in B$, we have
$$\psi(xe_n)=\htau(gxe_n)=\varphi(g_nx)e^{-n\beta}=\psi_n(x)e^{-n\beta}=\varphi(f_n x)e^{-n\beta}.$$
The correspondence $(f_n)\mapsto \psi$ is order-preserving by construction. 
\end{proof}

\begin{cor}\label{III} If 
$$e^{-\beta}\; s\mathchar`-\sum_{\xi\in \Lambda} \pi(\inpr{\xi}{\xi}_A)=1,$$
the von Neumann algebra $M$ is of type III. 
\end{cor}

\begin{proof} The assumption means that $Q$ is well-defined and $Q(1)=1$. 
Lemma \ref{trans} implies that $\varphi\circ P=P$, and the tail boundary has  
a faithful invariant probability measure. 
This would not be possible if the reduced flow of weights had dissipative part. 
\end{proof}

Recall that $\tau$ is a trace on $A$ defined by $\varphi\circ \Pi$. 
When $\tau$ is not a state, the projection $p_\infty$ is not zero and $B=B_0\oplus \C p_\infty$, 
where $B_0=\overline{\pi(A)}^w$. 
When we regard $p_\infty$ as an element in $\fZ(B)$, we can identify it as the characteristic function 
$\chi_{\{\infty\}}$ of a point at infinity in $X$. 
Let $X_0=X\setminus \{\infty\}$, and let $\mu_0=\chi_{X_0}\mu$. 
Then we have $\fZ(B_0)=L^\infty(X_0,\mu_0)$, $X=X_0\sqcup \{\infty\}$, and $\mu=\mu_0+(1-\|\tau\|)\delta_\infty$.

Let $\{a_i\}_{i\in I}$ be an approximate unit of $A$. 
Recall that we have $P(p_\infty)=0$. 
Let $P_0$ be a Markov operator acting on $\fZ(B_0)$ given by $P_0(f)=P(f)(1-p_\infty)$. 
Then we have 
$$P(f\oplus c p_\infty)=(P_0(f)\oplus \varphi_\infty(f)p_\infty),$$
where 
$$\varphi_\infty(f)=\frac{e^{-\beta}}{1-\|\tau\|}\sum_{\xi\in \Lambda}\lim_{i\in I}
\varphi(f\pi(\inpr{\xi}{(1-\Phi(a_i))\xi}_A)).$$
It is easy to see that that the Poisson boundary (resp. tali boundary) for $P$ is isomorphic 
to the Poisson boundary (resp. tail boundary) for $P_0$.  
Thus we get 

\begin{cor} \label{degenerate}
The reduced flow of weights $(\fZ(\tM),\theta|_{\fZ(\tM)})$ is identified with the tail boundary for $P_0$ 
with the time translation for positive $\beta$, and the inverse of the time translation for negative $\beta$. 
The center $\fZ(M)$ of $M$ is identified with the Poisson boundary for $P_0$. 
\end{cor}

\begin{cor}\label{degenerateIII} Assume that $\tau$ is not a state. 
If 
$$\;s\mathchar`-\lim_{F\Subset \Lambda} s\mathchar`-\lim_{i\in I}\sum_{\xi\in F}\pi(\inpr{\xi}{(\Phi(a_i))\xi}_A)$$
converges to a scalar, the von Neumann algebra $M$ is of type III. 
\end{cor}

\begin{proof} It suffices to show $\varphi\circ P_0=\varphi$ on $\fZ(B_0)$ as before. 
Indeed, for $f\in \fZ(B_0)$, we have 
\begin{align*}
\lefteqn{\varphi(P_0(f))=\varphi(P(f)(1-p_\infty))} \\
 &=\lim_{F\Subset \Lambda}\varphi(fQ_F(1-p_\infty))
 =\lim_{F\Subset \Lambda}\lim_{i\in I}\varphi(f\pi(\inpr{\xi}{\Phi(a_i)\xi}_A)),
\end{align*}
which shows that $\varphi\circ P_0$ is proportional to $\varphi|_{\fZ(B_0)}$. 
Since $P_0$ is unital, we get the statement. 
\end{proof}

\section{Examples}
We denote $\N_0=\{0\}\cup \N$, where $\N$ is the set of natural numbers. 
Let $A$ be a unital C$^*$-algebra and let $\alpha=(\alpha_1,\alpha_2,\ldots,\alpha_d)$ be a $d$-tuple 
of endomorphisms of $A$. 
We always assume $d\geq 2$. 
When every $\alpha_i$ is unital, we say that $\alpha$ is unital. 
As in \cite{Kak15}, we consider the Cuntz-Pimsner algebra arising from $\alpha$ though we do not 
assume that $\alpha$ is  unital. 
Let 
$$E=\ell^2(\{1,2,\ldots,d\})\otimes A=\bigoplus_{i=1}^d A,$$ 
which is a Hilbert C$^*$-module with 
$(\xi_i)\cdot a=(\xi_i a)$ and  
$$\inpr{\xi}{\eta}_A=\sum_{i=1}^d\xi_i^*\eta_i.$$
We denote by $\{e_i\}_{i=1}^d$ the canonical basis of $E$ over $A$. 
We define the left action $\Phi$ of $A$ on $E$ by $\Phi(a)(\xi_i)=(\alpha_i(a)\xi_i)$. 
When $\alpha$ is not unital, the left action is degenerate. 
The Cuntz-Toeplitz algebra $\cT_E$ has generators $\{T_i=T_{e_i}\}_{i=1}^d$ and $\Pi(A)$ satisfying the relations 
$$T_i^*T_j=\delta_{i,j}\Pi(1_A),$$
$$\Pi(a)T_i=T_i\Pi(\alpha_i(a)).$$
The Cuntz-Pimsner algebra $\cO_E$ is the quotient of $\cT_E$ by the ideal generated by 
$$\Pi(1_A)-\Pi(1_A)\sum_{i=1}^d T_iT_i^*.$$

With this setting, we keep using the same notation as in the previous sections except that 
we use the notation $(X,\mu)$, instead of $(X_0,\mu_0)$, for the point realization of $\fZ(B_0)$, 
as we mainly work on $B_0$, rather than $B$, in the rest of this paper. 
Now the two conditions in Theorem \ref{KMS0} are 
\begin{equation}\label{PF1}
e^{-\beta}d\tau(1_A)=1,
\end{equation}
\begin{equation}\label{PF2}
e^{-\beta}\tau(\sum_{i=1}^d\alpha_i(a))=\tau(a),\quad \forall a\in A. 
\end{equation}
When $\alpha$ is unital, the operator $\Pi(1_A)$ is the unit of $\cT_E$, and $\tau$ is a state. 
Thus Eq.(\ref{PF1}) holds if and only if $\beta=\log d$, and Eq.(\ref{PF2}) becomes
\begin{equation}\label{PF3}
\tau(\frac{1}{d}\sum_{i=1}^d\alpha_i(x))=\tau(x).
\end{equation}
(see \cite{Kai92}).
When $\alpha$ is not unital, the only constraint coming from Eq.(\ref{PF1}) is $\beta\leq \log d$,  
which makes the situation more flexible.

\begin{remark} 
In most of the examples of the correspondence $E$ with a dynamical origin in the literature 
(not necessarily coming from the above construction), it is quite often the case that 
an infinite type KMS state is unique and it exists only at $\beta=\log d$, where $d$ is a natural number determined by the dynamics being a sort of 
$d$-fold covering. 
As we saw in \cite{IKW07}, such a dynamics in the ergodic theoretical level is often isomorphic to 
the Bernoulli $d$-shift, and $M$ is the hyperfinite type III$_{\frac{1}{d}}$ factor. 
In particular, the reduced flow of weights is trivial. 
In the von Neumann algebra level (or ergodic theoretical level), such a model is equivalent to
the following example (see \cite{IKW07}). 
\end{remark}

\begin{example}  
Let $X=\{0,1,2,\ldots,d-1\}^{\N_0}$, and let $A=C(X)$. 
For $x=(x_n)\in X$ and $0\leq i\leq d-1$, we use the notation $ix$ for $(i,x_0,x_1,\ldots)\in X$. 
We introduce a $d$-tuple $\alpha=(\alpha_1,\alpha_2,\ldots,\alpha_d)$ of unital endomorphisms of $C(X)$ by 
$$\alpha_i(f)(x)=f((i-1)x).$$
Since $\alpha$ is unital, Eq.(\ref{PF1}) holds if and only if $\beta=\log d$. 
If we denote by $\mu$ the probability measure on $X$ corresponding to $\tau$, Eq.(\ref{PF2}) becomes 
$$\int_X\chi_{\{i\}}(x_0)f(x_1,x_2,\ldots)d\mu(x)=\frac{1}{d}\int_X f(x)d\mu(x).$$
This is possible if and only if $\mu$ is the Bernoulli measure $(\frac{1}{d}
\sum_{i=0}^{d-1}\delta_i)^{\otimes \N}$ with an equal weight, and we assume this. 
Then $Q$ is given by  
$$Q(g)(x)=\frac{1}{d}\sum_{i=0}^{d-1}g(ix),\quad g\in L^\infty(X,\mu),$$
and the Markov operator $P$ is the shift $P(f)((x_n)_{n=0}^\infty)=f((x_{n+1})_{n=0}^\infty)$. 
Let $(f_n)$ be a bounded harmonic sequence for $P$. 
Then 
$$f_k(x)=P^n(f_{n+k})(x)=f_{k+n}(x_{n},x_{n+1},\ldots),$$ 
and $f_k$ is measurable with respect to the tail $\sigma$-algebra, which is trivial for the Bernoulli measure.  
Thus $f_k$ is a constant function, and the tail boundary for $P$ is trivial. 
\end{example}

Next we give an easy example showing that the reduced flow of weights $\fT_M$ can be non-trivial. 

\begin{example} \label{III0} 
Let $X$ be a compact Hausdorff space, let $A=C(X)$, let $\fT$ be a homeomorphism of $X$, and 
let $\alpha_i(f)(x)=f(\fT^{-1}x)$ for $i=1,2,\ldots,d$. 
Then Eq(\ref{PF1}) holds if and only if $\beta=\log d$. 
Let $\mu$ be the probability measure corresponding to $\tau$. 
Then Eq(\ref{PF2}) holds if and only if $\mu$ is a $\fT$-invariant measure. 
Assume that $\mu$ is an ergodic invariant measure. 
Then $Q(g)(x)=g(\fT^{-1}x)$ and $P(f)(x)=(\fT x)$.  
In this case the tail boundary is identified by $(X,\mu,\fT)$, 
and the reduced flow of weights $(\fX_M,m_M,\fT_M)$ for the von Neumann algebra $M$ is identified with $(X,\mu,\fT)$.  
\end{example}

The above example in particular implies the following theorem. 

\begin{theorem} Let $\fT$ be a probability-measure-preserving ergodic transformation acting 
on a compact Hausdorff space, and let $d$ be an integer larger than 1. 
Then there exists a Cuntz-Pimsner algebra and an extreme $\log d$-KMS state of the gauge action such that 
the flow of weights for the factor arising from it via the GNS construction is the suspension flow over $\fT$ 
with a constant ceiling function $\log d$.  
\end{theorem}

Now we give non-unital examples.  
We will see that the following simple construction gives rise to surprisingly rich structure of 
infinite type KMS states. 
Recall that the structure of $K_\beta^\infty(\gamma)$ is determined by 
Theorem \ref{KMS0}, or alternatively, Eq.(\ref{PF1}) and Eq.(\ref{PF2}).

\begin{example}\label{non-unital} Let $X=\{0,1\}^{\N_0}$, and let $A=C(X)$. 
We often identify $A$ with the infinite tensor product $\otimes_{k=0}^\infty C(\{0,1\})$. 
For simplicity, we denote by $\chi_i\in C(\{0,1\})$ the characteristic function $\chi_{\{i\}}$ 
of the set $\{i\}$. 
Let $d/2\leq s\leq d$ be an integer and let $t=d-s$. 
We set  
$$\alpha_i(f)=\left\{
\begin{array}{ll}
\chi_0\otimes f  , &\quad 1\leq i\leq s,  \\
\chi_1\otimes f , &\quad s<i\leq d
\end{array}
\right..
$$
Since $\alpha=(\alpha_1,\alpha_2,\ldots,\alpha_d)$ is not unital, we need $\beta\leq \log d$ for the existence of 
an infinite type $\beta$-KMS state for the gauge action of $\cO_E$. 
Let $\mu$ be the probability measure of $X$ corresponding to $\frac{1}{\|\tau\|}\tau$ in Eq.(\ref{PF2}). 
Then Eq.(\ref{PF2}) is equivalent to 
\begin{equation}\label{PF4}
e^{-\beta}\int_{X}(s\chi_0\otimes f+t\chi_1\otimes f)d\mu(x)= \int_{X}f(x)d\mu(x). 
\end{equation}

When $d$ is even and $s=t=d/2$, Eq.(\ref{PF4}) becomes 
$$\frac{e^{-\beta}d}{2}\int_{X}(1\otimes f)d\mu(x)=\int_{X}fd\mu(x),$$
which is equivalent to the condition that $\beta=\log\frac{d}{2}$ and $\mu$ is a shift-invariant measure.  

When $t\neq s$, by setting $f=1$ in Eq.(\ref{PF4}), we get 
$$s\mu(C_0)+t\mu(C_1)=e^\beta,$$
$$\mu(C_0)+\mu(C_1)=1,$$
where for $(i_0,i_1,\ldots,i_{n-1})\in \{0,1\}^n$, we denote by $C_{i_0i_1\cdots i_{n-1}}$ the cylinder set 
$$\{(x_i)\in X;\; x_j=i_j,\quad 0\leq j\leq n-1\}.$$
Thus
$$\mu(C_0)=\frac{e^\beta-t}{s-t},\quad \mu(C_1)=\frac{s-e^\beta}{s-t}.$$
and $\log t\leq \beta \leq \log s$. 
The Bernoulli measure $b_p=\otimes_{k=0}^\infty((p\delta_0+(1-p)\delta_1)$ 
with $p=\frac{e^\beta-t}{s-t}$ satisfies Eq.(\ref{PF4}), and we see $K_\beta^\infty(\gamma)\neq \emptyset$ 
for $\log t\leq \beta\leq \log s$. 
\end{example}

We give more detailed accounts of this example treating three different cases separately. 
We first assume $s=d$. 
For $0<p<1$ and $y\in C_1$, we define a probability measure $m_{p,y}$ on $X$ by 
$$m_{p,y}=(1-p)\sum_{n=0}^\infty p^n\delta_{0^ny}.$$
When $\beta<\log d$, the measure $m_{\frac{e^\beta}{d},y}$ satisfies Eq.(\ref{PF4}). 
We denote by $\varphi_{\beta,y}$ the corresponding $\beta$-KMS state

\begin{theorem} Let the notation be as in Example \ref{non-unital}, and assume $s=d$.  
Then there exists an infinite type $\beta$-KMS state of the gauge action of $\cO_E$ if and only if 
$\beta\leq \log d$ (negative numbers are allowed though $\beta\neq 0$). 
\begin{itemize}
\item[(1)] For $\beta=\log d$, there exists a unique $\beta$-KMS state of infinite type and the corresponding 
factor $M$ is the Powers factor of type III$_{\frac{1}{d}}$. 
\item[(2)] For $\beta<\log d$, $\beta\neq 0$, we have $\ex K_\beta^\infty(\gamma)=\{\varphi_{\beta,y}\}_{y\in C_1}$, 
and all of them give rise to the hyperfinite II$_1$ factor. 
\end{itemize}
\end{theorem}

\begin{proof} (1) We first assume $\beta=\log d$. 
Then Eq.(\ref{PF4}) becomes 
$$\int_{X} \chi_0\otimes fd\mu(x)=\int_{X}f(x)d\mu(x),$$
and the Dirac measure $\delta_{0^\infty}$ is the only probability measure satisfying this condition. 
Since $B_0=\C$ in this case, the tail boundary is trivial, and the reduced flow of weights is trivial. 
Thus $M$ is of type III$_{\frac{1}{d}}$. 

(2) Assume $\beta < \log d$ now, and we set $p=\frac{e^\beta}{d}$. 
We claim that the affine space of probability measures on $X$ satisfying Eq.(\ref{PF4}) is affine isomorphic to 
the set of measures on $C_1$ with total mass $1-p$, and the isomorphism is given by restriction to $C_1$. 
Indeed, assume that $\mu$ satisfies Eq.(\ref{PF4}). 
Then we have $\mu(C_0)=p$ and $\mu(C_1)=1-p$. 
Repeated use of Eq.(\ref{PF4}) implies that the measure of the cylinder set $C_{0^n}$ is $p^n$. 
For a cylinder set $C_{i_0,i_1,\ldots,i_n}$ not of the above type, let $0\leq r\leq n$ be the smallest number with $i_r=1$. 
Then we have 
$$\mu(C_{i_0,i_2,\ldots,i_n})=p^r\mu(C_{i_r,i_{r+1},\ldots,i_n}),$$
which is determined by the restriction of $\mu$ to $C_1$. 
On the other hand, let $\mu_1$ be a measure of $C_1$ whose total mass is $1-p$. 
We may regard $\mu_1$ as a measure of $X$ by setting $\mu_1(Y)=\mu_1(Y\cap C_1)$. 
We construct a probability measure $\mu$ on $X$ by 
$$\mu=\sum_{n=0}^\infty p^n\delta_{0^n}\otimes \mu_1.$$
Then $\mu$ satisfies Eq(\ref{PF4}) and $\mu|C_1=\mu_1$, which proves the claim.

Thanks to the claim we see that the set of extreme points of the probability measures satisfying Eq.(\ref{PF4}) 
is given by $\{m_{p,y}\}_{y\in C_1}$, and hence we get $\ex K_\beta^\infty(\gamma)=\{\varphi_{\beta,y}\}_{y\in C_1}$. 
Let $M$ be the von Neumann algebra arising from $\varphi_{\beta,y}$. 
It remains to show that $M$ is a II$_1$ factor. 

We can identify $B_0$ with $\ell^\infty(\{0^ny\}_{n=0}^\infty)\cong \ell^\infty(\N_0)$. 
Then the Markov operator $P_0$ is given by the backward shift, and its tail boundary with time translation 
is isomorphic to the translation of $\Z$.  
Thus the reduced flow of weights is isomorphic to the translation of $\Z$, and $M$ is a semi-finite factor. 
To finish the proof, it suffices to show that $M$ is finite. 

We can identify $B$ with 
$$\ell^\infty(\{0^ny\}_{n=0}^\infty\cup\{\infty\}),$$
and we denote by $p_n$ the projection in $B$ corresponding to the set $\{0^ny\}$. 
Since $\pi(a)S_j=S_j\pi(\chi_0\otimes a)$ for $a\in A$, we have $p_nS_j=S_jp_{n+1}$. 
Since $S_jp_\infty=0$, we have  
$$p_\infty S_j=S_j-\sum_{n=0}^\infty p_nS_j=S_j-\sum_{n=0}^\infty S_jp_{n+1}=S_j(p_\infty+p_0)=S_jp_0.$$
Let 
$$\rho=e^{-\beta} p_\infty+\sum_{n=0}^\infty e^{n\beta}p_n,$$
which is a positive operator affiliated with $M$. 
Since $\sigma^\varphi_t(S_j)=e^{-\beta t\sqrt{-1}}S_j$, 
we get $\Ad \rho^{i t\sqrt{-1}} (S_j)=\sigma^\varphi_t(S_j)$, and $\sigma^\varphi_t=\Ad \rho^{\sqrt{-1}t}$. 
Thus $\varphi(\rho^{-1}\cdot )$ is a faithful normal semi-finite trace of $M$. 
Since 
\begin{align*}
\lefteqn{\varphi(\rho^{-1})=e^\beta\varphi(p_\infty)+\sum_{n=0}^\infty e^{-n\beta}\varphi(p_n)} \\
 &=e^\beta(1-\frac{e^\beta}{d})+\sum_{n=0}^\infty e^{-n\beta}(1-\frac{e^\beta}{d})(\frac{e^\beta}{d})^{n+1} 
 =e^\beta(1-\frac{e^\beta}{d})+(1-\frac{e^\beta}{d})\frac{e^\beta}{d}\sum_{n=0}^\infty \frac{1}{d^n} \\
 &=e^\beta\frac{d-e^\beta}{d-1}<\infty,
\end{align*}
it is a finite trace, and $M$ is finite. 
\end{proof}

We denote by $\sigma$ the shift of $X$, that is, $\sigma((x_n)_{n=0}^\infty)=(x_{n+1})_{n=0}^\infty$. 
We abuse the notation and we use the same symbol for the shift of $\tX=\{0,1\}^\Z$, that is, 
$\sigma((x_n)_{n\in \Z})=(x_{n+1})_{n\in \Z}$. 
Let $A_\infty=C(\tX)$, and let $A_n$ be the C$^*$-subalgebra of $A_\infty$ generated by 
the coordinate functions $\{x_k\}_{k=-n}^\infty$. 
We often identify $A$ with $A_0$ and regard $A$ as a subalgebra of $A_\infty$. 
Identifying probability measures with states for abelian C$^*$-algebras, 
the restriction of a probability measure on $A_\infty$ to $A$ makes sense. 

We assume that $0<t\leq e^\beta \leq s<d$ now, and define $r_0\in C(\{0,1\})$ by 
$$r_0=\frac{s}{e^\beta}\chi_0+\frac{t}{e^{\beta}}\chi_1,$$
and define $r\in A_\infty$ by $r(x)=r_0(x_0)$.  
Letting $\lambda=\frac{t}{s}$ and $q=\frac{\log s-\beta}{\log s-\log t}$, we have 
$0<\lambda\leq 1$, $0\leq q\leq 1$, and $r(x)=\lambda^{x_0-q}$. 

\begin{definition}
We denote by $\cM(\{0,1\}^\Z,\sigma,\lambda,q)$ the set of $\sigma$-quasi-invariant probability measures $\nu$ on 
$\{0,1\}^\Z$ satisfying 
\begin{equation}\label{qi}
\frac{d\nu\circ \sigma}{d\nu}(x)=\lambda^{x_0-q}, 
\end{equation} 
for $\nu$-a.e. $x$. 
\end{definition}

\begin{theorem}\label{shift} 
Let the notation be as in Example \ref{non-unital}, and assume that $0<t\leq e^\beta \leq s<d$. 
Let $r(x)$, $\lambda$, and $q$ be as above. 
Then the restriction of $\nu\in \cM(\tX,\sigma,\lambda,q)$ to $A$ gives rise to 
an affine isomorphism from $\cM(\tX,\sigma,\lambda,q)$ onto 
the set of probability measures $\mu$ on $X$ satisfying Eq.(\ref{PF4}). 
In particular $K_\beta^\infty(\gamma)$ is affine isomorphic to $\cM(\tX,\sigma,\lambda,q)$. 

Moreover, for the infinite type $\beta$-KMS state corresponding to $\mu$, the reduced flow of weights for 
the von Neumann algebra $M$ arising from it is isomorphic to the shift $\sigma$ on 
the measure space $(\tX,\nu)$.  
\end{theorem}

\begin{proof} (1) Assume that a probability measure $\nu$ on $\tX$ satisfies Eq.(\ref{qi}), and we set $\mu$ 
to be the restriction of $\nu$ to $A$. 
For $f\in A$, we have 
$$\int_{\tX}f(x)d\nu(x)=\int_{\tX}f\circ \sigma(x) d\nu\circ\sigma(x)=\int_{\tX}r(x)f\circ\sigma(x)d\nu(x),$$
which shows that $\mu$ satisfies Eq.(\ref{PF4}). 
Next we claim that $\nu$ is uniquely determined by $\mu$. 
To show the claim, it suffices to show that the restriction of $\nu$ to $A_n$ is determined by $\mu$. 
Indeed, for $f\in A_n$, we have 
\begin{align*}
\lefteqn{\int_{\tX}f(x)d\nu(x)=\int_{\tX}f\circ \sigma^n(x) d\nu\circ \sigma^n(x)} \\
 &= \int_{\tX}f\circ \sigma^n(x)r(x)r\circ\sigma(x)\cdots r\circ\sigma^{n-1}(x) d\nu(x)\\
 &=\int_{X}f(y_0,y_1,\ldots)r_0(y_0)r_0(y_1)\cdots r_0(y_{n-1})d\mu(y), 
\end{align*}
which shows the claim. 

Assume conversely that $\mu$ is a probability measure on $X$ satisfying Eq.(\ref{PF4}). 
We define a state $\omega_n$ of $A_n$ by 
$$\omega_n(f)=\int_{\{0,1\}^n\times X} f(x_{-n},x_{-n+1},\ldots)r_0(x_{-n})r_0(x_{-n+1})\cdots r_0(x_{-1})d\mu(x_{-n},x_{-n+1},\ldots).$$
Then Eq.(\ref{PF4}) implies that the restriction of $\omega_n$ to $A_{n-1}$ is $\omega_{n-1}$. 
Thus we can uniquely extend $\omega_n$ to $A_\infty$, and we denote by $\nu$ the corresponding probability 
measure on $\tX$. 
For $f\in A_n$, we have 
\begin{align*}
\lefteqn{\int_{\tX}f(x)d\nu\circ\sigma(x)=\int_{\tX} f\circ \sigma^{-1}(x)d\nu(x)=\omega_{n+1}(f\circ \sigma^{-1})} \\
 &=\int_{\{0,1\}^{n+1}\times X} f(x_{-n-1},x_{-n},\ldots)r_0(x_{-n-1})r_0(x_{-n})\cdots r_0(x_{-1})d\mu(x_{-n-1},x_{-n},\ldots) \\
 &=\int_{\{0,1\}^n\times X} r_0(x_0)f(x_{-n},x_{-n+1},\ldots)r_0(x_{-n})r_0(x_{-n+1})\cdots r_0(x_{-1})d\mu(x_{-n},x_{-n+1},\ldots) \\
 &=\omega_{n}(rf)=\int_{\tX}f(x)r(x)d\nu(x).
\end{align*}
Since this holds for any function $f$ in $\cup_{n=0}^\infty A_n$, we get Eq.(\ref{qi}). 
This establishes the one-to-one correspondence between $\nu$ and $\mu$. 

Assume $\mu$ and $\nu$ are related as above, and let $\varphi$ be the infinite type $\beta$-KMS state corresponding to $\mu$. 
Then we have $Q(g)=r_0\otimes g$, and thanks to Lemma \ref{trans} the Markov operator $P_0$ is determined by 
\begin{align*}
\lefteqn{\int_X P_0(f)(x)g(x)d\mu(x)} \\
 &=\int_X f(x)(r_0\otimes g)(x)d\mu(x)=\int_{\tX}fr g\circ \sigma d\nu=\int_{\tX} f\circ \sigma^{-1} g d\nu.
\end{align*}
Let $\cG_{[m,n]}$ be the $\sigma$-algebra generated by $\{x_i\}_{i=m}^n$. 
When $n=\infty$, we use the notation $\cG_{[m,\infty)}$.  
Identifying $B_0$ with $L^\infty(\tX,\cG_{[0,\infty)},\nu)$, 
we get 
$$P_0(f)=E(f\circ \sigma^{-1};\cG_{[0,\infty)}),$$
where the right-hand side is the conditional expectation of $f\circ \sigma^{-1}$ given $\cG_{[0,\infty)}$.

We show that the tail boundary for $P_0$ with time translation is isomorphic to $(\tX,\nu)$ with the shift $\sigma$. 
For $f\in L^\infty(\tX, \nu)$, we set 
$$f_n=E(f\circ \sigma^n;\cG_{[0,\infty)})\in B_0.$$ 
Let $0\leq k\leq n$ and $g\in B_0$. 
Note that the Radon-Nikodym derivative $\frac{d\nu\circ\sigma^{-k}}{d\nu}$ is $\cG_{[-k,-1]}$-measurable.
Then\begin{align*}
\lefteqn{\int_{\tX}f_n(x)g(x)d\nu(x)=\int_{\tX}f\circ\sigma^n(x)g(x)d\nu(x)} \\
 &=\int_{\tX}f\circ \sigma^{n-k}(x)g\circ \sigma^{-k}(x)d\nu\circ\sigma^{-k}(x)\\
 &=\int_{\tX}f\circ \sigma^{n-k}(x)g\circ \sigma^{-k}(x)\frac{d\nu\circ\sigma^{-k}}{d\nu}(x)d\nu(x)\\ 
 &=\int_{\tX}E(f\circ \sigma^{n-k};\cG_{[-k,\infty)})(x)
 g\circ \sigma^{-k}(x)\frac{d\nu\circ\sigma^{-k}}{d\nu}(x)d\nu(x) \\
 &= \int_{\tX}E(f\circ \sigma^{n-k};\cG_{[-k,\infty)})(x)
 g\circ \sigma^{-k}(x)d\nu\circ\sigma^{-k}(x)\\
 &=\int_{\tX}E(f\circ \sigma^{n-k};\cG_{[-k,\infty)})\circ\sigma^k(x)
 g(x)d\nu(x),
\end{align*}
which shows $f_n=E(f\circ \sigma^{n-k};\cG_{[-k,\infty)})\circ\sigma^k$. 
Thus 
$$P_0(f_{n+1})=P_0(E(f\circ\sigma^n;\cG_{[-1,\infty)})\circ \sigma)
=E(E(f\circ\sigma^n;\cG_{[-1,\infty)});\cG_{[0,\infty)})=f_n.$$
In a similar way, we can show $P_0(f_{n+1})=f_n$ for negative $n$ too, 
and we get $(f_n)\in S^\infty((X,\mu),P_0)$. 

On the other hand, let $(f_n)\in S^\infty((X,\mu),P_0)$, and we set 
$g_n=f_n\circ \sigma^{-n}$, which is $\cG_{[-n,\infty)}$-measurable. 
For $h\in L^\infty(\tX,\cG_{[-n,\infty)},\nu)$, we have 
\begin{align*}
\lefteqn{\int_{\tX} g_{n+1}h d\nu}\\
&=\int_Xf_{n+1}(x_{-n-1},\ldots)h(x_{-n},\ldots)
r_0(x_{-n-1})r_0(x_{-n})\cdots r_0(x_{-1})d\mu(x_{-n-1,\ldots}) \\
 &=\int_X P_0(f_{n+1})(x_{-n},\ldots)h(x_{-n},x_{-n+1},\ldots)
r_0(x_{-n})\cdots r_0(x_{-1})d\mu(x_{-n},\ldots) \\
 &=\int_{\tX} g_n h d\nu, 
\end{align*}
which shows $E(g_{n+1};\cG_{[-n,\infty)})=g_n$. 
Thus the martingale convergence theorem shows that $\{g_n\}_{n=0}^\infty$ converges in $L^2(\tX,\nu)$, 
and we denote by $f$ the limit. 
Since $\{g_n\}$ is bounded in $L^\infty(\tX,\nu)$, the function $f$ belongs to $L^\infty(\tX,\nu)$. 
The correspondence between $f$ and $(f_n)$ we have just constructed give an unital order preserving isometry between 
two Banach spaces $L^\infty(\tX,\nu)$ and $S^\infty((X,\mu_0),P_0)$. 
Moreover, it intertwines the shift and time translation.  
Thus the reduced flow of weights for the von Neumann algebra arising from $\varphi$ is 
isomorphic to $(\tX,\nu,\sigma)$.   
\end{proof}

Note that for $\lambda=1$, the set $\cM(\tX,\sigma,\lambda,q)$ is nothing but the set of shift-invariant probability 
measures on $\tX$. 

\begin{cor}\label{d/2} Let the notation be as in Example \ref{non-unital}.   
We assume that $d$ is an even number larger than 2, and $s=t=d/2$. 
Then there exists an infinite type $\beta$-KMS state of the gauge action $\gamma$ of $\cO_E$ if and only if 
$\beta=\log\frac{d}{2}$. 
There exists an affine isomorphism between $K_{\log\frac{d}{2}}^\infty(\gamma)$ and the set of shift-invariant probability 
measures on $\{0,1\}^\Z$. 
For each shift-invariant probability measure $\nu$, the reduced flow of weights for the von Neumann algebra $M$ 
arising from the corresponding KMS state is isomorphic to the shift on $(\{0,1\}^\Z,\nu)$.   
\end{cor}

We assume $0<t<s$ now. 
Then $0<\lambda<1$, and the Radon-Nikodym derivative $r(x)=\lambda^{x_0-p}$ in Eq.(\ref{qi}) is not trivial. 
For $n\in \Z$ and $x\in \tX$, we set 
$$c_n(x)=\left\{
\begin{array}{ll}
\sum_{j=0}^{n-1}x_j-qn , &\quad n\geq 1 \\
0 , &\quad n=0  \\
-\sum_{j=1}^{-n}x_{-j}-qn , &\quad n\leq -1
\end{array}
\right.,
$$
which is an additive cocycle of $\sigma$ in the sense that $c_n$ satisfies $c_{m+n}=c_m+c_n\circ \sigma^m$. 
For $\nu\in \cM(\tX,\sigma,\lambda,q)$, we have 
$$\frac{d\nu\circ \sigma^n}{d\nu}(x)=\lambda^{c_n(x)}.$$
Thus if $\nu$ has an atom at $z\in \tX$, it has an atom at $\sigma^n(z)$ for any $n\in \Z$ and 
$$\nu(\{\sigma^n(z)\})=\nu\circ \sigma^n(\{z\})=\lambda^{c_n(z)}\nu(\{z\}).$$ 
In particular, if $z$ is a periodic point with period $n$, we have 
$c_n(z)=0$, and 
$$\frac{1}{n}\sum_{j=0}^{n-1}z_j=q=\frac{\log s-\beta}{\log s-\log t}.$$
With this observation, we can determine the atomic measures in $\ex \cM(\tX,\sigma,\lambda,q)$ 
as follows. 

Assume that $y\in \tX$ is a periodic point with minimal period $n$. 
For
$$\beta=\log s-(\log s-\log t)\frac{1}{n}\sum_{k=0}^{n-1}y_k,$$
we define a probability measure $\nu_y$ on $\tX$ by 
$$\nu_y=\frac{1}{\sum_{k=0}^{n-1}\lambda^{c_k(y)}}\sum_{k=0}^{n-1}\lambda^{c_k(y)}\delta_{\sigma^k(y)}.$$ 
Then $\nu_y\in \ex\cM(\tX,\sigma,\lambda,q)$, and we denote by 
$\varphi_y$ the corresponding KMS state in $\ex K_\beta^\infty(\gamma)$.  

Let $z\in \tX$ be an aperiodic sequence satisfying 
\begin{equation}\label{C}
\sum_{n\in \Z}\lambda^{c_n(z)}<\infty,
\end{equation}
and define a probability measure $\nu_{\beta,z}$ on $\tX$ by 
$$\nu_{\beta,z}=\frac{1}{\sum_{n\in \Z}\lambda^{c_n(z)}}\sum_{n\in \Z}\lambda^{c_n(z)}\delta_{\sigma^n(z)}.$$ 
Then $\nu_{\beta,z}\in \ex\cM(\tX,\sigma,\lambda,q)$, and we denote by $\varphi_{\beta,z}$ 
the corresponding KMS state in $\ex K_\beta^\infty(\gamma)$. 
Note that $\varphi_{\beta,z}$ depends only on the orbit of $z$. 
The conditions (\ref{C}) holds if $\beta$ satisfies 
\begin{equation}\label{sufficient}
\limsup_{n\to\infty}\frac{1}{n}\sum_{k=1}^nz_{-k}<\frac{\log s-\beta}{\log s-\log t}<
\liminf_{n\to\infty}\frac{1}{n}\sum_{k=1}^nz_k.
\end{equation}

Note that for $q=0,1$, we have singletons 
$$\cM(\tX,\sigma,\lambda,0)=\{\delta_{0^\infty}\},\quad \cM(\tX,\sigma,\lambda,1)=\{\delta_{1^\infty}\}.$$  

\begin{cor} Let the notation be as in Example \ref{non-unital}, and  assume $0<t<s<d$. 
Then there exists an infinite type $\beta$-KMS state of the gauge action of $\cO_E$ if and only if 
$\log t\leq \beta\leq \log s$ for $t\neq 1$ and $0< \beta\leq \log s$ for $t=1$. 
\begin{itemize}
\item[(1)] Let $y\in \tX$ be a periodic point with minimal period $n$, and let 
$$\beta=\log s-(\log s-\log t)\frac{1}{n}\sum_{k=0}^{n-1}y_k.$$
Then $\varphi_y\in \ex K_\beta^\infty(\gamma)$, and the corresponding factor is the Powers factor of type 
III$_{e^{-n\beta}}$, where 
$$e^{-n\beta}=\prod_{k=0}^{n-1}\frac{1}{s^{1-y_k}t^{y_k}}.$$
In particular, 
\begin{itemize}
\item[(i)]  For $\beta=\log s$, there exists a unique infinite type $\beta$-KMS state $\varphi_{0^\infty}$, 
and the corresponding factor is the Powers factor of type III$_\frac{1}{s}$.  
\item[(ii)]  For $t\neq 1$ and $\beta=\log t$, there exists a unique infinite type $\beta$-KMS state 
$\varphi_{1^\infty}$, and the corresponding factor is the Powers factor of type III$_\frac{1}{t}$.  
\end{itemize}
\item[(2)] Let $z\in\{0,1\}^{\Z}$ be an aperiodic sequence  
and assume that $\beta$ satisfies (\ref{C}). 
Then $\varphi_{\beta,z}\in \ex K_\beta^\infty(\gamma)$, and the corresponding factor is semi-finite.  
The set $\{\varphi_{\beta,y}\}_y$ exhausts all the semi-finite factor states in $\ex K_\beta^\infty(\gamma)$. 
\end{itemize}
\end{cor} 

\begin{problem} Determine the type of $M$ for $\varphi_{\beta,z}$. 
\end{problem}

\begin{remark} In the case of $\log t<\beta<\log s$, the KMS state corresponding to the Bernoulli measure $b_p$ 
on $X$ with $p=\frac{e^\beta-t}{s-t}$ is not extreme, and the corresponding von Neumann algebra $M$ is semi-finite. 
Indeed, the reduced flow of weights in this case 
is isomorphic to the two-sided Bernoulli shift with the measure 
$$\nu=\bigotimes_{k=-\infty}^\infty \nu_k,$$
where 
$$\nu_k=\left\{
\begin{array}{ll}
p\delta_0+(1-p)\delta_1 , &\quad k\geq 0 \\
e^{-\beta}sp\delta_0+e^{-\beta}t(1-p)\delta_1 , &\quad k<0
\end{array}
\right..
$$
This is known to be dissipative (see \cite{H82}). 
Thus in principle we should have ergodic decomposition 
$$\nu=\int_{D_\beta/\sigma} \nu_{\beta,z}dm(z),$$
where $D_\beta$ is the set of $z$ satisfying (\ref{C}), 
though we do not know the explicit formula of the measure $m$.  
\end{remark}

Next we show that there exists a diffuse measure in $\ex \cM(\tX,\sigma,\lambda,q)$ for some $\beta$. 
For an aperiodic sequence $z\in \tX$ \textit{not} satisfying (\ref{C}) and $n\in \N$, we set 
$$\nu_{\beta,z,n}=\frac{1}{\sum_{|k|\leq n}\lambda^{c_k(z)}}\sum_{|k|\leq n}\lambda^{c_{k}(z)}\delta_{\sigma^k(z)}.$$
Then 
$$\|\nu_{\beta,z,n}\circ\sigma-\lambda^{c_1}\nu_{\beta,z,n}\|
=\frac{\lambda^{c_{-n}(z)}+\lambda^{c_{n+1}(z)}}{\sum_{|k|\leq n}\lambda^{c_k(z)}}.$$
When this converges to 0 as $n$ tends to infinity, we can get a measure 
in $\cM(\tX,\sigma,\lambda,q)$ by taking a cluster point of 
$\{\nu_{\beta,z,n}\}_{n=1}^\infty$. 
Such a measure concentrates on the orbit closure $Z=\overline{\{\sigma^k(z)\}_{k\in \Z}}$, 
which is a subshift. 
The most tractable class of sequences $z\in \tX$ satisfying the above condition is those with bounded 
$\{c_{k}(z)\}_k$. 

\begin{proposition}\label{diffuse} Let the notation be as above, and let $Z\subset \tX$ be 
a closed shift-invariant subspace without periodic points satisfying  
$$\sup_{k\in \Z,\;z\in Z}|c_k(z)|<\infty.$$ 
Then there exists a diffuse measure in $\ex \cM(\tX,\sigma,\lambda,q)$ supported by $Z$. 
\end{proposition}

\begin{proof} Since $\sup_{n\in \Z}\|c_n\|_{C(Z)}<\infty$, the $\sigma$-cocycle $\{c_n\}$ restricted to $Z$ 
is a coboundary, that is, there exists $h\in C(Z)$ satisfying 
$c_n(z)=h\circ \sigma^n(z)-h(z)$ for any $z\in Z$ and $n\in \Z$ 
(see \cite[Theorem 2.3]{W00}). 
Let $\nu_0$ be an ergodic shift-invariant probability measure on $Z$. 
Then  
$$\frac{1}{\int_Z\lambda^{h(z)}d\nu_0(z)}\lambda^h\nu_0$$
regarded as a probability measure on $\tX$ belongs to $\ex \cM(\tX,\sigma,\lambda,q)$, and it is 
diffuse because $Z$ has no periodic points.  
\end{proof}

Minimal subshifts arising from a special class of substitutions of two alphabets 
(e.g. Pisot substitutions) provide examples satisfying the assumption in Proposition \ref{diffuse} 
and with a unique ergodic measure $\nu_0$. 
The reader is referred to \cite[Chapter 5]{Q10} for the substitution dynamical systems. 
More precisely, we need a substitution $\zeta$ with the following property: 
the $\zeta$-matrix has an eigenvalue whose modulus is less than or equal to 1 
(see \cite[Definition 5.4]{Q10} for the definition).  
For example, let $Z$ be the two-sided substitution dynamical system arising from the Thue-Morse sequence, 
which is given by the substitution rule $0\to01$, $1\to10$ (see \cite[page 128]{Q10}). 
Let $\beta=\log\sqrt{st}$. 
In this case, we have $q=\frac{1}{2}$ and an easy uniform estimate  
$$-1\leq c_k(z)\leq 1,\quad \forall z\in Z,\;\forall k\in \Z.$$

\begin{cor} Let the notation be as in Example \ref{non-unital}, and  assume $0<t<s<d$. 
There exists a KMS state in $\ex K_{\log\sqrt{st}}^\infty(\gamma)$ with the corresponding factor $M$ 
of type III$_0$. 
\end{cor}

\begin{problem} Does there exist a measure in $\ex \cM(\{0,1\}^\Z,\sigma,\lambda,q)$ 
not equivalent to an invariant measure? 
Such a measure would give rise to a KMS state in $\ex K_\beta^\infty(\gamma)$ with the corresponding 
flow of weights having no invariant measure. 
\end{problem}

In order to construct a measure as above, an obvious trial we could make would be to apply the argument 
before Proposition \ref{diffuse} to a sequence $z\in \tX$ with $\{c_n(z)\}_{n}$ growing (or oscillating) slowly. 
For example, we could examine sequences discussed in \cite{G72}. 
Let $\{k(n)\}_{n=1}^\infty$ be a sequence of natural numbers with $k(n)\geq 3$, 
and set $l(0)=1$ and $l(n)=k(1)k(2)\cdots k(n)$ for $n \geq 1$. 
We define words $a(n),b(n)\in \{0,1\}^{\{0,1,\cdots,l(n)-1\}}$ by 
$a(0)=1$, $b(0)=0$, and  
$$a(n)=a(n-1)b(n-1)a(n-1)^{k(n)-2},$$
$$b(n)=a(n-1)b(n-1)^{k(n)-1},$$
for $n\geq 1$. 
Then we have 
$$\sum_{j=0}^{l(n)-1}a(n)_j=\frac{1}{2}l(n)+\frac{1}{2}\prod_{j=1}^n(k(j)-2),$$
$$\sum_{j=0}^{l(n)-1}b(n)_j=\frac{1}{2}l(n)-\frac{1}{2}\prod_{j=1}^n(k(j)-2).$$
We define $z\in \tX$ by setting $z_j=0$ for negative $j$, and $z_j=a(n)_j$ for $j\geq 0$, 
where $n$ is a sufficiently large number. 
Let $q=1/2$. 
The above computation implies  
$$c_{l(n)}(z)=\frac{1}{2}\prod_{j=1}^n(k(j)-2),$$
$$c_{2l(n)}(z)=0.$$
Defining a probability measure $\nu_n$ on $\tX$ by 
$$\nu_n=\frac{1}{\sum_{j=0}^{l(n)-1}\lambda^{c_j(z)}}\sum_{j=0}^{l(n)-1}\lambda^{c_j(z)}\delta_{\sigma^j(z)},$$
we get 
$$\|\nu_n\circ \sigma-\lambda^{c_1}\nu_n\|=\frac{1+\lambda^{c_{l(n)}(z)}}{\sum_{j=0}^{l(n)-1}\lambda^{c_j(z)}}
\leq \frac{2}{n}.$$
Thus any cluster point of $\{\nu_n\}_{n=0}^\infty$ belongs to $\cM(\tX,\sigma,\lambda,1/2)$, 
which is supported by the minimal subshift 
$Z=\overline{\{\sigma^n(z)\}_{n=0}^\infty}\setminus \{\sigma^n(z)\}_{n=0}^\infty$. 
The subshift $(Z,\sigma)$ is a Toeplitz flow, and it comes from a substitution 
when $k(n)$ is a constant function (see \cite{GJ00}). 
It is uniquely ergodic if and only if $\sum_{n=1}^\infty \frac{1}{k(n)}=\infty$. 
In fact when it is not uniquely ergodic, any $q$ with
$$\frac{1}{2}\leq q\leq \frac{1}{2}+\frac{1}{2}\prod_{j=1}^\infty (1-\frac{2}{k(j)})$$
works as well.

We end this paper with a noncommutative variant of Example \ref{non-unital}. 
Let $A$ be the gauge invariant CAR algebra, that is the fixed point algebra of 
the CAR algebra $\otimes_{k=0}^\infty M_2(\C)$ under the product action 
$$\bigotimes_{k=0}^\infty\Ad\left(
\begin{array}{cc}
1 &0  \\
0 &e^{t\sqrt{-1}} 
\end{array}
\right)$$
of $\T$. 
For $0\leq p\leq 1$, we denote by $\tau_p$ the restriction of the Powers state 
$$\bigotimes_{k=0}^\infty \Tr(\left(
\begin{array}{cc}
p &0  \\
0 &1-p 
\end{array}
\right)\cdot 
)$$
to $A$. 
It is known that $\{\tau_p\}_{0\leq p\leq 1}$ exhausts all extreme trace states of $A$ 
(see \cite[Appendix]{R80}). 
We denote 
$$e_0=\left(
\begin{array}{cc}
1 &0  \\
0 &0 
\end{array}
\right),\quad 
e_1=\left(
\begin{array}{cc}
0 &0  \\
0 &1 
\end{array}
\right).$$

\begin{example}\label{GICAR} Let $A$ be the gauge invariant CAR algebra. 
We set 
$$\alpha_i(x)=\left\{
\begin{array}{ll}
e_0\otimes x , &\quad 1\leq i\leq s \\
e_1\otimes x , &\quad s<i\leq d
\end{array}
\right..
$$
In this case Eq.(\ref{PF1}) and Eq.(\ref{PF2}) become $\tau(1_A)=\frac{e^\beta}{d}$ and 
\begin{equation}\label{PF5}
e^{-\beta}\tau(se_0\otimes x+te_1\otimes x)=\tau(x). 
\end{equation}

Assume that $d$ is even and $s=t=\frac{d}{2}>1$ first. 
Then an infinite type $\beta$-KMS state exists if and only if $\beta=\log \frac{d}{2}$, 
and $K_{\log \frac{d}{2}}^\infty(\gamma)$ is isomorphic to the tracial simplex of $A$ as affine spaces. 
For an extreme infinite type KMS state $\varphi$, the trace $\tau$ is extreme too and 
$\fZ(B_0)$ is trivial, and so is the tail boundary of $P_0$.  
Thus every extreme KMS state of infinite type gives rise to the Powers factor of type III$_{\frac{2}{d}}$.  

Assume now that $t<s\leq d$. 
We can express $\tau$ as 
$$\tau=\frac{e^\beta}{d}\int_{[0,1]}\tau_rdm(r),$$
with a unique probability measure $m$ on $[0,1]$, and  Eq.(\ref{PF5}) is
$$\int_{[0,1]}\tau_r(se_0\otimes x+te_1\otimes x)dm(r)=e^{\beta}\int_{[0,1]}\tau_r(x)dm(r). $$
Since the left-hand side is 
$$\int_{[0,1]}\big(sr+t(1-r)\big)\tau_r(x)dm(r),$$ 
and the measure $m$ concentrates on $r$ satisfying 
$sr+t(1-r)=e^\beta$, and we get 
$$\tau=\frac{e^\beta}{d}\tau_{\frac{e^\beta-t}{s-t}}.$$
Thus an infinite type $\beta$-KMS state for the gauge action exists if and only if $\log t\leq \beta\leq \log s$, 
$\beta\neq 0$, and $K_\beta^\infty(\gamma)$ is a singleton for each $\beta$. 
Every $\beta$-KMS state of infinite type gives rise the Powers factor of type III$_{e^{-\beta}}$.  
\end{example}

\end{document}